\documentclass[11pt]{article}
\usepackage{latexsym,subfigure,float}
\usepackage{comment}
\usepackage{amsmath,amsthm,amsfonts,amssymb,graphicx,epsfig,latexsym,color}
\usepackage{epsf,enumerate}

\title{A Bers-like proof of the existence of train tracks for free
  group automorphisms}
\author{Mladen Bestvina\thanks{The
    author gratefully acknowledges the support by the National
    Science Foundation. }} \date{January 12, 2011}

\newtheorem{thm}{Theorem}

\newtheorem{cor}[thm]{Corollary}
\newtheorem{prop}[thm]{Proposition}
\newtheorem*{main}{Theorem}{}
{}

\theoremstyle{remark}

\newtheorem{definition}[thm]{Definition}
\newtheorem{remark}[thm]{Remark}

\newtheorem*{definition*}{Definition}
\newtheorem*{remark*}{Remark}

\def\R{{\mathbb R}}

\def\X{{\mathcal X}}

\begin{document}

\maketitle

\begin{quote}
\begin{center}
{\it To Mike Davis on the occasion of his 60th birthday}
\end{center}
\end{quote}

\begin{abstract}
Using Lipschitz distance on Outer space we give another proof of the
train track theorem.
\end{abstract}

\section{Introduction}

An elegant proof of Thurston's classification of surface
homeomorphisms \cite{thurston:nielsen} was given by Bers
\cite{Bers}. Given a surface homeomorphism $\Phi$ the proof proceeds
by studying the associated displacement function $\tilde\Phi$ on
Teichm\"uller space $\mathcal T$ with respect to Teichm\"uller metric,
i.e. the function
$$\tilde\Phi(x)=d(x,\Phi(x))$$
There are three possibilities:
\begin{itemize}
\item (elliptic) $\inf\tilde\Phi=0$ and the infimum is realized. In
  this case $\tilde\Phi$ has a fixed point and it is not hard to show
  that $\Phi$ is isotopic to a homeomorphism of finite order.
\item (hyperbolic) $\inf\tilde\Phi>0$ and the infimum is realized. In
  this case Bers proceeds to show that $\Phi$ is isotopic to a
  pseudo-Anosov homeomorphism.
\item (parabolic) The infimum is not realized. In this case Bers shows
  that $\Phi$ is reducible.
\end{itemize}

In this paper we carry out the Bers' argument in the context of
$Out(F_n)$ and Outer space. The role of Teichm\"uller metric is played
by the Lipschitz metric, a direct analog of Thurston's Lipschitz
metric on Teichm\"uller space \cite{thurston:stretch}. This metric is
not symmetric and one must carefully choose the order of the two
points when measuring distance. The result is an alternate proof of
the train track theorem \cite{BH}.

\begin{main}
Every irreducible automorphism $\Phi$ has a
topological representative which is a train track map.
\end{main}

\noindent
{\bf Acknowledgements.} The author is indebted to Ric Wade for
reading the paper carefully and making helpful comments.

\section{Outer space and Lipschitz metric}

In this section we review basic definitions and set some notation.

A {\it graph} $\Gamma$ is a finite cell complex of dimension $\leq
1$. A {\it core graph} is a graph $\Gamma$ with all vertices of
valence $\geq 2$. A {\it rose} in rank $n$ is the wedge $R_n$ of $n$
circles. A {\it marking} of a graph $\Gamma$ is a homotopy equivalence
$f:R_n\to\Gamma$. A {\it metric} on $\Gamma$ is an assignment $\ell$
of positive lengths $\ell(e)$ to the edges $e$ of $\Gamma$ such that
the sum is 1. If $\alpha$ is an immersed loop in $\Gamma$ we define
the {\it length} of $\alpha$ with respect to the metric $\ell$ as the
sum $\ell(\alpha)$ of the lengths of edges of $\Gamma$ crossed by
$\alpha$, with multiplicity. If $\alpha$ is not immersed, we define
$\ell(\alpha)$ as the length of the immersed loop homotopic to
$\alpha$ (or 0 if $\alpha$ is nullhomotopic). We may view $\Gamma$ as
a geodesic metric space with each edge $e$ having length $\ell(e)$. A
{\it direction} at $x\in\Gamma$ is a germ of isometric embeddings
$d:[0,\epsilon)\to\Gamma$ with $d(0)=x$. Thus most points of $\Gamma$
  have two directions, and the number of directions at a vertex is the
  valence. Directions can be viewed as analogs of unit tangent
  vectors. If $\phi:\Gamma\to\Gamma'$ is a map which is linear on edges
  and $\phi(x)=x'$,
  then $\phi$ induces a map $\phi_*$ from the set of directions at $x$ to
  the set of directions at $x'$ (unless the slope of $\phi$ is 0 on an
  edge containing $x$).

Recall that Culler-Vogtmann's Outer space $\X_n$ in rank $n$ \cite{CV} is the
set of equivalence classes of triples $(\Gamma,f,\ell)$ where
\begin{itemize}
\item $\Gamma$ is a core graph,
\item $f:R_n\to\Gamma$ is a marking, and
\item $\ell$ is a metric on $\Gamma$.
\end{itemize}

Two such triples $(\Gamma,f,\ell)$ and $(\Gamma',f',\ell')$
are equivalent if there is a homeomorphism $\phi:\Gamma\to\Gamma'$
such that
\begin{itemize}
\item $\phi f\simeq f'$, and
\item for all loops $\alpha$ in $\Gamma$,
  $\ell(\alpha)=\ell'(\phi(\alpha))$. 
\end{itemize}

If $\Gamma$ and $\Gamma'$ have no vertices of valence 2, then $\phi$
must induce a bijection between the edges of $\Gamma$ and of $\Gamma'$
and $\ell(e)=\ell'(\phi(e))$ for every edge $e\subset\Gamma$. Every
triple $(\Gamma,f,\ell)$ is equivalent (for $n\geq 2$) to some
$(\Gamma',f',\ell')$ so that $\Gamma'$ has no vertices of valence 2,
by ``unsubdividing'' and assigning the sum of the lengths of
subdivision edges to the newly created edges.

There are three ways of defining the topology on $\X_n$, all yielding
the same topology.

\begin{itemize}
\item $\X_n$ can be decomposed into open simplices corresponding to
  fixing the marking and varying lengths on $\Gamma$. This gives
  $\X_n$ the structure of a ``complex of simplices with missing
  faces''. Take the weak
  topology with respect to the collection of these simplices with
  missing faces.
\item There is an embedding $\X_n\hookrightarrow \R^{\mathcal S}$
  where $\mathcal S$ is the set of nontrivial conjugacy classes in
  $F_n$, or equivalently, the set of immersed loops in $R_n$. The
  embedding is given by
$$[(\Gamma,f,\ell)]\mapsto (\alpha\mapsto \ell(f(\alpha)))$$
Now take the subspace topology.
\item A neighborhood of $[(\Gamma,f,\ell)]$ is determined by
  $\epsilon>0$ and consists of classes $[(\Gamma',f',\ell')]$ such
  that there is a map $\phi:\Gamma\to\Gamma'$ with $\phi f\simeq f'$
  and such that $\phi$ is $<(1+\epsilon)$-Lipschitz.
\end{itemize}

$Out(F_n)$ acts on $\X_n$ on the right as follows. Let $\Phi\in
Out(F_n)$. We may view $\Phi$ as a homotopy equivalence (defined up to
homotopy) $\Phi:R_n\to R_n$. Then
$$[(\Gamma,f,\ell)]\Phi=[(\Gamma,f\Phi,\ell)]$$

The third definition of the topology on $\X_n$ can be promoted to a
(nonsymmetric) metric.

Let $[(\Gamma,f,\ell)],[(\Gamma',f',\ell')]\in\X_n$. Consider maps
$\phi:\Gamma\to\Gamma'$ so that
\begin{itemize}
\item $\phi f\simeq f'$, and
\item $\phi$ is linear on edges.
\end{itemize}

Call any map $\phi$ that satisfies these two conditions {\it a
  difference of markings}. 
Let $\sigma(\phi)$ denote the maximal slope of $\phi$. 

Observe that if $\alpha$ is any loop in $\Gamma$ then
$$\ell'(\phi(\alpha))\leq\sigma(\phi)\ell(\alpha)$$

The following result is due to Tad White (unpublished). A proof
appears in \cite{FM}, but for completeness we include a proof.

\begin{prop}\label{greengraph}
Let $\phi_0:\Gamma\to\Gamma'$ be a difference of markings. Then
$$\inf\{\sigma(\phi)\mid \phi\simeq\phi_0:\Gamma\to\Gamma'\mbox{ is a
  difference of
  markings}\}=\sup_\alpha\frac{\ell'(\phi_0(\alpha))}{\ell(\alpha)}$$
and moreover both $\inf$ and $\sup$ are realized.
\end{prop}

\begin{proof}
That $\inf\geq\sup$ follows from the observation just before the
proposition. Arzela-Ascoli implies that $\inf$ is realized. Let
$\phi:\Gamma\to\Gamma'$ be a difference of markings that realizes
$\inf$. By $\Delta=\Delta(\phi)$ denote the union of all edges of
$\Gamma$ on which $\phi$ has slope equal to $\sigma(\phi)$. We may
also assume that $\Delta$ is minimal possible. Therefore $\Delta$ is a
core graph (any edges with a valence 1 vertex can be removed by a
small homotopy of $\phi$, by moving the image of the valence 1 vertex
in the direction that lowers the slope on the edge of $\Delta$
containing it).

Let $v$ be a vertex of $\Delta$ and consider a {\it turn}, i.e. pair
of distinct directions $\{d,d'\}$ out of $v$ in $\Delta$. We say that
this turn is {\it legal} if $\phi_*(d)\neq \phi_*(d')$, and it is {\it
  illegal} if $\phi_*(d)=\phi_*(d')$. A path or a loop in $\Delta$ is
{\it legal} if it crosses only legal turns.  Note that the loop
$\phi(\alpha)$ is immersed if and only if $\alpha$ is legal. Also
observe that there is an equivalence relation on the set of directions
(within $\Delta$) out of each vertex $v\in\Delta$ where $d\sim d'$ if
and only if $\phi_*(d)=\phi_*(d')$. A turn $\{d,d'\}$ is legal if and
only if $d\not\sim d'$. We will call equivalence classes {\it gates}.

If $\Delta$ contains a vertex with only one gate, then a small
homotopy as above would reduce $\Delta$. Therefore there are at least
two gates at each vertex. It follows that every legal edge path in
$\Delta$ can be extended in a legal fashion. In particular, $\Delta$
admits a legal loop $\alpha$. By construction
$\ell'(\phi(\alpha))=\sigma(\phi)\ell(\alpha)$, and thus $\alpha$
realizes the $\sup$ and equality holds as claimed.
\end{proof}

\begin{remark}
The legal loop $\alpha$ can be chosen to cross each edge of $\Delta$
at most twice. Therefore, the $\sup$ can be calculated by taking the
maximum over a finite collection of loops in $\Gamma$. Furthermore,
$\ell'(\phi_0(\alpha))$ does not depend on $\phi_0$, only on the homotopy
class of $\phi_0$, so the quantity in the statement can be easily
calculated. 
\end{remark}

To simplify notation, we will replace $[(\Gamma,f,\ell)]$ with
$\Gamma$. Denote by
$$\sigma(\Gamma,\Gamma')$$
the quantity in the statement of Proposition \ref{greengraph}.

\begin{definition}
Let $\Gamma,\Gamma'\in\X_n$. A difference of markings
$\phi:\Gamma\to\Gamma'$ is {\it optimal} if
$\sigma(\phi)=\sigma(\Gamma,\Gamma')$.  The {\it tension (sub)graph}
$\Delta=\Delta_\phi\subset\Gamma$ with respect to an optimal map
$\phi$ is the union of edges on which the slope of $\phi$ equals
$\sigma(\phi)$.

The tension graph is equipped with a {\it train
track structure} as in the proof of Proposition \ref{greengraph}: a
turn $\{d,d'\}$ is legal if $\phi_*(d)\neq \phi_*(d')$ and otherwise
it is illegal. Directions (in $\Delta$) at a vertex break up into
equivalence classes, called {\it gates}, so that $\{d,d'\}$ is illegal if
and only if $d,d'$ are in the same gate.
\end{definition}

\begin{definition}
$$d(\Gamma,\Gamma')=\log\sigma(\Gamma,\Gamma')$$
\end{definition}

\begin{prop}
\begin{itemize}
\item $d(\Gamma,\Gamma')\geq 0$ with equality only when
  $\Gamma=\Gamma'$. 
\item $d(\Gamma,\Gamma'')\leq d(\Gamma,\Gamma')+d(\Gamma',\Gamma'')$.
\item $Out(F_n)$ acts on $\X_n$ by isometries: $d(\Gamma\cdot \Phi,\Gamma'
\cdot  \Phi)=d(\Gamma,\Gamma')$. 
\end{itemize}
\end{prop}

\begin{proof}
If $\sigma(\Gamma,\Gamma')<1$ then the volume of the image (i.e. the
sum of the lengths of images of edges) of an optimal map
$\phi:\Gamma\to\Gamma'$ is $<1$, contradicting the fact that $\phi$ is
surjective. Likewise, if $\sigma(\Gamma,\Gamma')=1$ then $\phi$ must
be an isometry.

The second statement follows by composing optimal maps and homotoping
rel vertices to a map linear on edges. The third statement is obvious.
\end{proof}

But note that in general $d(\Gamma,\Gamma')\neq
d(\Gamma',\Gamma)$. See \cite{AB}. It can also be shown that
$d:\X_n\times\X_n\to [0,\infty)$ is continuous.

\section{Trichotomy}\label{trichotomy}

Fix an automorphism $\Phi\in Out(F_n)$ and consider the {\it
  displacement function} 
$$\tilde\Phi:\X_n\to [0,\infty)$$
given by
$$\Gamma\mapsto d(\Gamma,\Gamma\cdot\Phi)$$
There are three possibilities, as follows:
\begin{itemize}
\item (elliptic) $\inf\tilde\Phi=0$ and it is realized.
\item (hyperbolic) $\inf\tilde\Phi>0$ and it is realized.
\item (parabolic) $\inf\tilde\Phi$ is not realized.
\end{itemize}
We consider these cases separately.

\subsection{$\Phi$ is elliptic}

An example is pictured below.

\begin{figure}[ht]
\centerline{\input{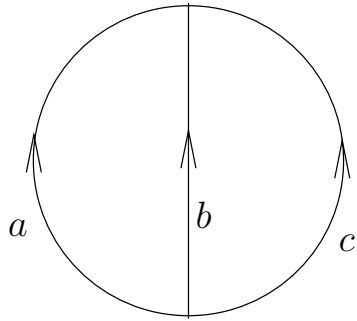}}
\caption{\label{elliptic} $\Phi(a)=\overline b$, $\Phi(b)=\overline
  c$, $\Phi(c)=\overline a$. All 3 edges have length $1/3$, the
  tension graph is all of $\Gamma$ and all nondegenerate turns are
  legal. $\Phi$ has order 6. Bar denotes the edge with opposite orientation.}
\end{figure}

Here we are assuming that $\Phi$ has a fixed point,
i.e. $\Gamma\cdot\Phi=\Gamma$ for some $\Gamma\in\X_n$. If $f$ is the
marking of $\Gamma$, we have $f\Phi\simeq \phi f$ for an isometry
$\phi:\Gamma\to\Gamma$. Since (for $n\geq 2$) isometries of $\Gamma$
have finite order, it follows that for some $k>0$ we have $\phi^k=id$
and hence $f=\phi^k f\simeq f\Phi^k$, so $\Phi^k$ is (homotopic to)
the identity, i.e. $\Phi$ has finite order.

\subsection{$\Phi$ is hyperbolic}

Suppose that $\inf \tilde\Phi$ is realized on $\Gamma\in\X_n$. Let
$\log\lambda=d(\Gamma,\Gamma\cdot\Phi)=\inf\tilde\Phi>0$, so
$\lambda>1$. Let $\Delta\subseteq \Gamma$ be the tension graph with its
train track structure, with respect to an optimal map $\phi$.

An example of a hyperbolic automorphism is given in Figure
\ref{hyperbolic}. 

\begin{figure}[H]
\centerline{\input{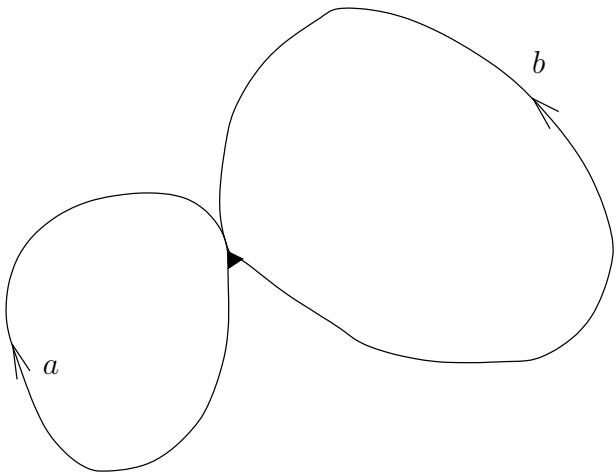}}
\caption{\label{hyperbolic} $\Phi(a)=ab$, $\Phi(b)=bab$.
Lengths of edges and $\sigma(\phi)=\lambda$ are computed from the
equations $\ell(a)+\ell(b)=1$, $\lambda \ell(a)=\ell(a)+\ell(b)$ and $\lambda
\ell(b)=\ell(a)+2\ell(b)$, i.e. $\lambda=\frac{3+\sqrt{5}}2$,
$\ell(a)=\frac{3-\sqrt{5}}2$, $\ell(b)=\frac{\sqrt{5}-1}2$. The tension graph is
all of $\Gamma$. The gates are indicated by the little triangle at the
vertex: they are $\{\overline a,\overline b\}, \{a\}, \{b\}$, where we
adopt the convention that $a$ represents the initial direction of the
edge $a$, while $\overline a$ represents the terminal direction of
$a$, and similarly for $b$.}
\end{figure}

\begin{prop}\label{tt}
After an arbitrarily small perturbation of $\Gamma$ that preserves the
condition that $d(\Gamma,\Gamma\cdot\Phi)=\log\lambda$ there is
an optimal map $\phi:\Gamma\to\Gamma\cdot\Phi$ such that
\begin{itemize}
\item
  $\phi(\Delta)\subseteq\Delta$,
\item $\phi$ sends edges of
$\Delta$ to
legal paths, and 
\item $\phi_*$ sends legal turns to
legal turns.
\end{itemize}
\end{prop}

A 2-gate vertex $v$ is allowed to be mapped to a non-vertex, with the
directions in each gate mapping to one of two directions out of
$\phi(v)$, i.e. the two directions out of a non-vertex are regarded as
forming a legal turn. 

\begin{proof}
Fix an optimal map $\phi:\Gamma\to\Gamma\cdot\Phi$. The {\it
  complexity} of $\Delta$ is the pair
$(rank(H_1(\Delta)),-rank(H_0(\Delta))$. In the moves that follow the
complexity is never increased, and is often decreased. Note that if
$\Delta$ is a core graph and $\Delta'\subset\Delta$ is a proper core
subgraph, then the complexity of $\Delta'$ is strictly smaller than
that of $\Delta$.

Suppose that all vertices of $\Delta$ have $\geq 2$ gates (in
particular, $\Delta$ is a core graph), that 
$\phi(\Delta)\not\subseteq\Delta$ and let $e$ be an edge of $\Delta$
with $\phi(e)\not\subset\Delta$. Perturb the metric on $\Gamma$ by
scaling by $\mu>1$ on the edges of $\Delta$ and scaling down on the
edges in the complement of $\Delta$, maintaining volume 1. Denote the
new metric graph $\Gamma'$. The tension graph $\Delta'$ for the new
pair $\Gamma'\to\Gamma'\cdot\Phi$ (using the same map made linear on
edges) is contained in $\Delta$ and does
not contain $e$. Note that the slope on some edge must be
$\lambda$ by the minimality assumption. Continuing in this way we
obtain a perturbation of $\Gamma$ where
$\phi(\Delta)\subseteq\Delta$. If during the process we encounter
$\Delta$ with a 1-gate vertex, we perturb the map as in the proof of
Proposition \ref{greengraph} with the effect that $\Delta$ is replaced
by a smaller graph.

If $\phi$ maps an edge $e$ of $\Delta$ over an illegal turn, first perturb
by folding the illegal turn (this means identify initial
segments of small length $\epsilon>0$ in the two edges and rescale the
metric so the volume is 1; the map $\phi$ naturally induces a map on
the quotient). After homotoping rel vertices to a map linear on edges,
we see that an edge (induced by $e$) drops out of $\Delta$ and
complexity decreases. Again it may be necessary to remove 1-gate
vertices. 

Now suppose that $\phi_*$ maps a legal turn to an illegal
turn. Perturb by folding the illegal turn. This converts the legal
turn to an illegal turn and it either lowers the number
$$\sum_v [\max(0,G(v)-2)]$$
where $G(v)$ is the number of gates at $v$, or else it introduces a
1-gate vertex. In the latter case the subsequent perturbation of
$\phi$ lowers the complexity of $\Delta$.
At the end of the process we have $\phi$ and $\Delta$
satisfying the conclusion. 
\end{proof}

\begin{cor}
Let $\Gamma$ realize $\inf\tilde\Phi=\log\lambda$. Then
$d(\Gamma,\Gamma\cdot \Phi^k)=k\log\lambda$ for any $k=1,2,3,\cdots$. 
\end{cor}

\begin{proof}
After perturbing as in the proof of Proposition \ref{tt}, the
statement follows by observing that for a legal loop $\alpha$ the loop
$\phi(\alpha)$ is also legal, and iterating we find that the length of
$\phi^k(\alpha)$ is equal to $\lambda^k\ell(\alpha)$. By continuity,
this is true before perturbing as well.
\end{proof}

\begin{remark}
It is easy to see that if $\Phi$ is any automorphism and
$\phi:\Gamma\to\Gamma\cdot \Phi$ an optimal map satisfying the
conclusions of Proposition \ref{tt},
then $\tilde\Phi$ achieves minimum at $\Gamma$. Indeed, as above we
have $d(\Gamma,\Gamma\cdot\Phi^k)=k d(\Gamma,\Gamma\cdot\Phi)$ for
$k>0$ and if
$\Gamma'\in \X_n$ then 
\begin{eqnarray*}
d(\Gamma,\Gamma\cdot\Phi^k)\leq
d(\Gamma,\Gamma')+d(\Gamma',\Gamma'\cdot\Phi^k)+
d(\Gamma'\cdot\Phi^k,\Gamma\cdot\Phi^k)\leq \\
d(\Gamma,\Gamma')+k
d(\Gamma',\Gamma'\cdot\Phi)+d(\Gamma',\Gamma)
\end{eqnarray*}
The first and last
terms are independent of $k$ so by dividing by $k$ and taking the
limit as $k\to\infty$ we see that $d(\Gamma,\Gamma\cdot\Phi)\leq
d(\Gamma',\Gamma'\cdot\Phi)$.
\end{remark}

For another example of an automorphism of this type consider
$\Phi:F_3\to F_3$ given by $\Phi(a)=ab$, $\Phi(b)=bab$ and
$\Phi(c)=cw$ where $w$ is any word in $a$ and $b$. Let $\Gamma$ be the
rose with the metric on $a,b$ as in Figure \ref{hyperbolic} but scaled
by $t>0$, and let the length of $c$ be $1-t$. When $t>0$ is
sufficiently small (depending on $w$) $\Gamma$ realizes the minimal
displacement $\log\lambda$ with the same $\lambda$ as in Figure
\ref{hyperbolic}.

\subsection{$\Phi$ is parabolic}

One example of a parabolic automorphism, with $\inf\tilde\Phi=0$, is
shown in Figure \ref{parabolic1}.

\begin{figure}[ht]
\centerline{\input{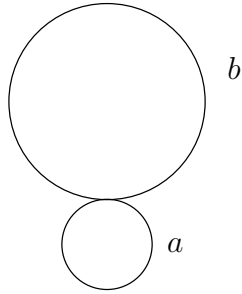}}
\caption{\label{parabolic1} $\Phi(a)=a$, $\Phi(b)=ab$.
Let the length of $a$ be $t>0$ and the length of $b$ is $1-t$. Then
the slope on $a$ is 1 and the slope on $b$ is $\frac 1{1-t}$. Thus as
$t\to 0$ the displacement converges to 0. The tension graph is the
loop $b$.}
\end{figure}

For another example, consider $\Phi(a)=ab$, $\Phi(b)=bab$,
$\Phi(c)=cad$, $\Phi(d)=dcad$. On $a$ and $b$ take the metric from
Figure \ref{hyperbolic} scaled by a parameter $t>0$, and on $c$ and
$d$ take the same metric (with $c$ corresponding to $a$ and $d$ to
$b$) scaled by $1-t$. The tension graph consists of $c$ and $d$, with
slope converging to $\lambda=\frac{3+\sqrt{5}}2$ as $t\to 0$. In this
example $\inf\tilde\Phi=\log\lambda>0$ is not realized. (One way to
see this is to show that the (combinatorial) length of $\Phi^k(c)$
grows like $k\lambda^k$.)

Fix a sequence $\Gamma_k$ such that $$d(\Gamma_k,\Gamma_k\cdot\Phi)\to
D=\inf\tilde\Phi$$ but $D$ is not realized.

For $\theta>0$ denote by $\X_n(\theta)$ the subspace of $\X_n$
consisting of marked metric graphs $\Gamma$ where every nontrivial
loop has length $\geq\theta$ (the {\it $\theta$-thick part of
  $\X_n$}). Then $\X_n(\theta)$ is an $Out(F_n)$-invariant closed subset
  on which $Out(F_n)$ acts cocompactly.

\begin{prop}\label{thin}
For any $\theta>0$ there are only finitely many $\Gamma_k$ with
$\Gamma_k\in\X_n(\theta)$. 
\end{prop}

\begin{proof}
Suppose not, and after passing to a subsequence assume that
$\Gamma_k\in \X_n(\theta)$ for every $k$. By cocompactness, there are
$\Psi_k\in Out(F_n)$ such that $\Gamma_k\cdot\Psi_k\to\Gamma_\infty$
(after a subsequence). Therefore, we have
\begin{eqnarray*}
&d(\Gamma_\infty\cdot\Psi_k^{-1},\Gamma_\infty\cdot\Psi_k^{-1}\Phi)\\
\leq& d(\Gamma_\infty\cdot\Psi_k^{-1},\Gamma_k)+d(\Gamma_k,\Gamma_k\cdot\Phi)+d(\Gamma_k\cdot\Phi,\Gamma_\infty\cdot\Psi_k^{-1}\Phi)\\
=&d(\Gamma_\infty,\Gamma_k\cdot\Psi_k)+d(\Gamma_k,\Gamma_k\cdot\Phi)+d(\Gamma_k\cdot\Psi_k,\Gamma_\infty) 
\end{eqnarray*}

and
hence $$d(\Gamma_\infty\cdot\Psi_k^{-1},\Gamma_\infty\cdot\Psi_k^{-1}\Phi)\to
D$$ 
(since $d(\Gamma_k\Psi_k,\Gamma_\infty)\to 0$ and
$d(\Gamma_\infty,\Gamma_k\Psi_k)\to 0$). In other words,
$$d(\Gamma_\infty\cdot\Psi_k^{-1}\Phi^{-1}\Psi_k,\Gamma_\infty)\to D$$
Note that Arzela-Ascoli implies that there are only finitely many
$\Psi\in Out(F_n)$ such that $d(\Gamma_\infty\Psi,\Gamma_\infty)\leq
D+1$ (the set of $e^{D+1}$-Lipschitz maps
$\Gamma_\infty\to\Gamma_\infty$ is compact and nearby maps are
homotopic, so such maps represent only finitely many homotopy classes).
It follows that, after a subsequence,
$\Psi_k^{-1}\Phi^{-1}\Psi_k$ is a constant sequence and
$$d(\Gamma_\infty\cdot\Psi_k^{-1}\Phi^{-1}\Psi_k,\Gamma_\infty)=D$$
i.e. the displacement of $\Gamma_\infty\cdot\Psi_k^{-1}$ under $\Phi$ is $D$,
contradicting our assumption that $\inf$ is not realized.
\end{proof}

For $\epsilon>0$ the $\epsilon$-small subspace $\Gamma^\epsilon$ of
$\Gamma$ is the union of all essential (not necessarily immersed)
loops of length $\leq\epsilon$ (but note that $\Gamma^\epsilon$ may
not be a subgraph). There is $\epsilon_n>0$ such that for any
$\Gamma\in\X_n$ the subspace $\Gamma^{\epsilon_n}$ is always proper
(i.e. not equal to $\Gamma$, e.g. there is always an edge of length
$\geq 1/(3n-3)$, assuming no valence 2 vertices, so
$\epsilon_n<1/(3n-3)$ works). Moreover, there is a bound $B_n$ to the
length of any chain of proper core subgraphs.

\begin{prop}\label{10}
For large $k$ any optimal map $\phi:\Gamma_k\to\Gamma_k\cdot\Phi$ leaves a
nonempty 
proper core subgraph invariant up to homotopy (and so $\phi$ is
homotopic to a possibly non-optimal map that maps this core subgraph
to itself).
\end{prop}

\begin{proof}
Let
$\theta=\epsilon_n/(e^{D+1})^{B_n}$. By Proposition \ref{thin} eventually
$\Gamma_k\not\in\X_n(\theta)$. Choose $k$ so large that in addition the
displacement of $\Gamma_k$ is $<D+1$. Set
$\delta_i=\epsilon_n/e^{(D+1)^i}$, $i=0,1,2,\cdots,B_n$. Then 
$$\Gamma_k^{\delta_0}\supset
\Gamma_k^{\delta_1}\supset\cdots\supset\Gamma_k^{\delta_{B_n}}$$ form
a chain of homotopically nontrivial proper subspaces (not necessarily
subgraphs, but abstract graphs) of length $B_n+1$, so there must be
some $i$ so that $\Gamma_k^{\delta_i}$ and $\Gamma_k^{\delta_{i+1}}$
have the same core. By definition, an optimal map must send
$\Gamma_k^{\delta_{i+1}}$ into $\Gamma_k^{\delta_i}$, so the common
core is mapped to itself up to homotopy.
\end{proof}

\section{Train tracks}

In this section we complete the proof of the train track theorem. We
first make the standard definitions; the definition of a train track
structure is as before, but does not require a metric on the graph. 

\begin{definition}
Let $\Gamma$ be a marked graph with marking $f:R_n\to\Gamma$. We say
that $\phi:\Gamma\to\Gamma$ represents $\Phi\in Out(F_n)$ if $\phi
f\simeq f\Phi$. The automorphism $\Phi$ is {\it reducible} if there is
some $\phi:\Gamma\to\Gamma$ that represents $\Phi$ and leaves a
homotopically nontrivial (i.e. not a forest) proper subgraph
invariant. Otherwise $\Phi$ is {\it irreducible}.
\end{definition}

\begin{definition}
A {\it train track structure} on a core graph $\Gamma$ is an
equivalence relation on the set of directions at every vertex of
$\Gamma$ with at least two equivalence classes (called gates) at every
vertex. As before, a turn $\{d,d'\}$ is {\it illegal} if $d\sim d'$,
and otherwise it is {\it legal}. An immersed loop or a path is {\it
  legal} if it takes only legal turns.
\end{definition}

\begin{definition}
Let $\Phi\in Out(F_n)$ be irreducible.  A map $\phi:\Gamma\to\Gamma$
representing $\Phi$ is a {\it train track map} if it sends each edge
to a nondegenerate immersed path and the following two equivalent
conditions are satisfied.
\begin{enumerate}[(i)]
\item
There
is a train track structure on $\Gamma$ such that $\phi$ sends edges to
legal paths and if $v$ is a vertex of $\Gamma$ then either
\begin{itemize}
\item $\phi(v)=w$ is a vertex and inequivalent directions at $v$ map to
  inequivalent directions at $w$, or
\item $\phi(v)$ is not a vertex, $v$ has two gates, and all directions in
  one gate fold to one direction at $\phi(v)$, while all directions in
  the other gate fold to the other direction at $\phi(v)$.
\end{itemize}
\item There is a loop $\alpha$ in $\Gamma$ such that every iterate
  $\phi^k(\alpha)$ is immersed, $k=1,2,3,\cdots$.
\end{enumerate}
\end{definition}

To prove that (i) implies (ii) it suffices to choose any legal loop
for $\alpha$, since (i) guarantees that legal loops map to legal
loops. For the converse, first note that the iterates of $\alpha$
cross every edge of $\Gamma$ (otherwise the union of the edges crossed
by the iterates of $\alpha$ would be a homotopically nontrivial proper
$\phi$-invariant subgraph). Thus iterated images of edges are
immersed. Define a train track structure on $\Gamma$ by declaring
$d\sim d'$ if $\phi_*^k(d)=\phi_*^k(d')$ for some $k\geq 1$. The
reader may now easily check that the conditions in (i) hold.

For example, any simplicial isomorphism $\phi:\Gamma\to\Gamma$ is a
train track map.

\begin{remark}
For a given $\phi:\Gamma\to\Gamma$ there may be more than one
invariant (i.e. satisfying (i)) train track structure. For example, in
Figure 2 we could take the gates to be $\{\overline a,\overline b\}$
and $\{a,b\}$. In the above paragraph we constructed the invariant
train track structure with the {\it minimal} collection of illegal
turns. 
\end{remark}

The discussion in Section
\ref{trichotomy} proves the following theorem.

\begin{main}\cite{BH}
Every irreducible automorphism $\Phi$ is
represented by a train track map $\phi:\Gamma\to\Gamma$.
\end{main}

There is an additional caveat. In \cite{BH} train track maps always
send vertices to vertices. This is very useful, and luckily it is not
hard to achieve. Let $\phi_t$, $t\in [0,1]$, be a homotopy of $\phi=\phi_0$
that moves each $\phi(v)$ that is not a vertex to an endpoint of the
edge containing $\phi(v)$.  We also insist that during the homotopy
the order of the images of vertices in the same edge does not change
(i.e. there are no collisions, nor ``uncollisions'') until the very
last moment when several images of vertices may arrive at the same
vertex. It is not hard to see that the images of legal loops under
$\phi_t$ are unchanged (and they are still legal loops). There may be
edges that map to points under $\phi_1$; collapse all such edges
iteratively (i.e. after a collapse there may be new such edges that
are then collapsed). We obtain a new map $\phi':\Gamma'\to\Gamma'$
representing $\Phi$.  Any legal loop in $\Gamma$ induces a loop in
$\Gamma'$ whose $\phi'$-iterates are immersed, so $\phi'$ is a train
track map that sends vertices to vertices.

Finally, as in \cite{BH}, one can put a metric on $\Gamma'$ by solving
linear equations e.g. as in Figure 2. The array of lengths is a
positive left eigenvector of the transition matrix $M$ for $\phi'$ whose
$ij$-entry is the number of times $\phi'(e_j)$ crosses $e_i$ in either
direction. The irreducibility of $\Phi$ implies the irreducibility of
$M$, so by the Perron-Frobenius theory $M$ has a unique (positive)
eigenvalue with an associated positive eigenvector.

\begin{remark}
When $\phi:\Gamma\to\Gamma$ sends vertices to vertices and edges to
nondegenerate immersed paths, finding an invariant train track
structure is algorithmic (when it exists). One forms a (finite)
directed graph whose vertices are the directions at the vertices of
$\Gamma$ and a directed edge from $d$ to $d'$ when
$d'=\phi_*(d)$. Thus each vertex of the directed graph has one
outgoing edge and following outgoing edges eventually ends in a
periodic orbit (i.e. a cycle).  Then define $d\sim d'$ if $d$ and $d'$
are based at the same vertex and their forward iterates eventually
coincide. For example, in Figure 2 we have $a\mapsto a$, $b\mapsto b$
and $\overline a\mapsto\overline b\mapsto\overline b$, so $\overline
a\sim\overline b$ is the only nontrivial equivalence.
\end{remark}

Along the same lines one can give a proof of the existence of relative
train track maps \cite{BH} representing any given $\Phi\in
Out(F_n)$. One works in a relative Outer space, where all graphs
contain a fixed subgraph $\Gamma_0$ and on which
$\phi:\Gamma_0\to\Gamma_0$ partially representing $\Phi$ has already
been constructed. The edges of $\Gamma_0$ are assigned length 0. The
strategy of the absolute case applies here as well. The details may
appear elsewhere.

\nocite{fm2}
\bibliography{./ref}

\end{document}